\documentclass[11pt,a4paper]{amsart}
\usepackage{amsmath,amsthm,amsfonts,graphicx,amssymb,eufrak,mathrsfs,calc}
\usepackage[numbers,sort&compress]{natbib}
\usepackage[lmargin=32mm,rmargin=32mm,tmargin=35mm,bmargin=35mm]{geometry}
\renewcommand{\baselinestretch}{1.5}
\theoremstyle{plain}
\newtheorem{theorem}{Theorem}[section]
\newtheorem{lemma}[theorem]{Lemma}

\theoremstyle{definition}
\newtheorem{conjecture}[theorem]{Conjecture}

\newcommand{\seps}{\ensuremath{\mathbb{S}}}
\newcommand{\half}{\ensuremath{\protect\tfrac{1}{2}}}
\newcommand{\ceil}[1]{\ensuremath{\protect\lceil#1\rceil}}
\newcommand{\CEIL}[1]{\ensuremath{\protect\left\lceil#1\right\rceil}}

\newcommand{\thmlabel}[1]{\label{thm:#1}}
\newcommand{\thmref}[1]{Theorem~\ref{thm:#1}}
\newcommand{\twothmref}[2]{Theorems~\ref{thm:#1} and \ref{thm:#2}}

\newcommand{\lemlabel}[1]{\label{lem:#1}}
\newcommand{\lemref}[1]{Lemma~\ref{lem:#1}}

\newcommand{\eqnlabel}[1]{\label{eqn:#1}}
\newcommand{\eqnref}[1]{\eqref{eqn:#1}}
\newcommand{\twoeqnref}[2]{\eqref{eqn:#1} and \eqref{eqn:#2}}
\newcommand{\figlabel}[1]{\label{fig:#1}}
\newcommand{\figref}[1]{Figure~\ref{fig:#1}}
\newcommand{\seclabel}[1]{\label{sec:#1}}
\newcommand{\secref}[1]{Section~\ref{sec:#1}}
\newcommand{\conjlabel}[1]{\label{con:#1}}
\newcommand{\conjref}[1]{Conjecture~\ref{con:#1}}    

\begin{document}

\title{Contractibility and the Hadwiger Conjecture}

\author{David~R.~Wood}
\address{\newline Department of Mathematics and Statistics
\newline The University of Melbourne
\newline Melbourne, Australia}
\email{D.Wood@ms.unimelb.edu.au}


\date{\today}

\begin{abstract}
Consider the following relaxation of the Hadwiger Conjecture: For each $t$ there exists $N_t$ such that every graph with no $K_t$-minor admits a vertex partition into $\ceil{\alpha t+\beta}$ parts, such that each component of the subgraph induced by each part has at most $N_t$ vertices. The Hadwiger Conjecture corresponds to the case  $\alpha=1$, $\beta=-1$ and $N_t=1$. Kawarabayashi and Mohar [\emph{J.~Combin.\ Theory Ser.~B}, 2007] proved this relaxation with $\alpha=\frac{31}{2}$ and $\beta=0$ (and $N_t$ a huge function of $t$). This paper proves this relaxation with $\alpha=\frac{7}{2}$ and $\beta=-\frac{3}{2}$. The main ingredients in the proof are: (1) a list colouring argument due to Kawarabayashi and Mohar, (2) a recent result of Norine and Thomas that says that every sufficiently large $(t+1)$-connected graph contains a $K_t$-minor, and (3) a new sufficient condition for a graph to have a set of edges whose contraction increases the connectivity.
\end{abstract}

\maketitle

\newpage
\section{Introduction}
\seclabel{Intro}


In 1943, \citet{Hadwiger43} made the following conjecture, which is widely considered to be one of the most important open problems in graph theory; see \citep{Toft-HadwigerSurvey96} for a survey\footnote{All graphs in this paper are undirected, simple and finite. Let $G$ be a graph. The vertex set and edge set of $G$ are denoted by $V(G)$ and $E(G)$. For $v\in V(G)$, let $N_G(v):=\{w\in V(G):vw\in E(G)\}$. If $X\subseteq V(G)$ then $G[X]$ denotes the subgraph induced by $X$. 
If $vw$ is an edge of $G$ then $G/vw$ is the graph obtained from $G$ by contracting $vw$; that is, the edge $vw$ is deleted and the vertices $v$ and $w$ are identified. A \emph{minor} of $G$ is a graph that can be obtained from a subgraph of $G$ by contracting edges. 
A \emph{$k$-colouring} of $G$ is a function that assigns one of at most $k$ colours to each vertex of $G$, such that  adjacent vertices receive distinct colours. $G$ is \emph{$k$-colourable} if $G$ admits a $k$-colouring. }.


\medskip\noindent\textbf{Hadwiger Conjecture.}
Every graph with no $K_t$-minor is $(t-1)$-colourable.
\smallskip


The Hadwiger Conjecture is true\footnote{If $G$ has no $K_1$-minor then $V(G)=\emptyset$ and $G$ is $0$-colourable. If $G$ has no $K_2$-minor then $E(G)=\emptyset$ and $G$ is $1$-colourable. If $G$ has no $K_3$-minor then $G$ is a forest, which is $2$-colourable.  \citet{Hadwiger43} and \citet{Dirac52} independently proved that if $G$ has no $K_4$-minor (so-called \emph{series-parallel} graphs) then $G$ is $3$-colourable. The Hadwiger Conjecture with $t=5$ implies the Four-Colour Theorem, since planar graphs contain  no $K_5$-minor. In fact, \citet{Wagner37} proved that the Hadwiger Conjecture with $t=5$ is equivalent to the Four-Colour Theorem, and is therefore true \citep{Gonthier-NAMS08,RSST97}.  \citet{RST-Comb93} proved that the Hadwiger Conjecture with $t=6$ also is a corollary of the Four-Colour Theorem.} for $t\leq 6$. \citet{Kostochka82,Kostochka84} and \citet{Thomason84,Thomason01} independently proved that for some constant $c$, every graph $G$ with no $K_t$-minor has a vertex of degree at most $ct\sqrt{\log t}$ (and this bound is best possible). It follows that $G$ is $ct\sqrt{\log t}$-colourable. This is the best known such upper bound. In particular, the following conjecture is unsolved:


\smallskip\noindent\textbf{Weak Hadwiger Conjecture.}
For some constant $c$, 
every graph with no $K_t$-minor is $ct$-colourable.
\smallskip

This conjecture motivated \citet{KawaMohar-JCTB07} to prove the following relaxation; see \citep{Kawarabayashi-CPC08} for a recent extension to graphs with no odd $K_t$-minor.

\begin{theorem}[\citet{KawaMohar-JCTB07}]
\thmlabel{RelaxedHadwiger} 
For each $t\in\mathbb{Z}^+$ there exists $N_t\in\mathbb{Z}^+$ such that every graph with no $K_t$-minor admits a vertex partition into $\CEIL{\frac{31}{2}t}$ parts, and each connected component of the subgraph induced by each part has at most $N_t$ vertices.
\end{theorem}

With $N_t=1$ the vertex partition in \thmref{RelaxedHadwiger} is a colouring. So 
\thmref{RelaxedHadwiger} is a relaxation of the Weak Hadwiger Conjecture. It would be interesting to improve the bound of  $\frac{31}{2}t$  in \thmref{RelaxedHadwiger}. Indeed,  \citet{KawaMohar-JCTB07} write, 
\begin{quote}
``The $\frac{31}{2}t$ bound can be improved slightly by fine-tuning parts of the proof in \citep{BKMM}. 
However, new methods would be needed to go below $10t$.''
\end{quote}
The main contribution of this paper is to improve  $\frac{31}{2}$  in \thmref{RelaxedHadwiger} to $\frac{7}{2}$. 

\begin{theorem}
\thmlabel{Main}
For each $t\in\mathbb{Z}^+$ there exists $N_t\in\mathbb{Z}^+$ such that every graph with no $K_t$-minor admits a vertex partition into $\CEIL{\frac{7t-3}{2}}$ parts, and each connected component of the subgraph induced by each part has at most $N_t$ vertices.
\end{theorem}

There are three main ingredients to the proof of \thmref{Main}. The first ingredient is a list colouring argument due to  \citet{KawaMohar-JCTB07}, which is described in  \secref{ListColouring}. The second ingredient is a sufficient condition for a graph to have a set of edges whose contraction increases the connectivity. This condition generalises previous results by \citet{Mader-DM88}, and is presented in \secref{Contractibility}. The third ingredient, the ``new methods'' alluded to above, is the following recent result by \citet{NorineThomas}.

\begin{theorem}[\citet{NorineThomas}]
\thmlabel{NT}
For each $t\in\mathbb{Z}^+$ there exists $N_t\in\mathbb{Z}^+$ such that every $(t+1)$-connected graph with at least $N_t$ vertices has a $K_t$-minor.
\end{theorem}


\section{List Colouring}
\seclabel{ListColouring}

A key tool in the proofs of \twothmref{RelaxedHadwiger}{Main} is the notion of list colouring, independently introduced by \citet{Vizing76} and \citet{ERT80}. A \emph{list-assignment} of a graph $G$ is a function $L$ that assigns to each vertex $v$ of $G$ a set $L(v)$ of colours. $G$ is \emph{$L$-colourable} if there is a colouring of $G$ such that the colour assigned to each vertex $v$ is in $L(v)$. $G$ is \emph{$k$-choosable} if $G$ is $L$-colourable for every list-assignment $L$ with $|L(v)|\geq k$ for each vertex $v$ of $G$. If $G$ is $k$-choosable then $G$ is also $k$-colourable---just use the same set of $k$ colours for each vertex. See \citep{Woodall-ListColouringSurvey} for a survey on list colouring. 

As well as being of independent interest, list colourings enable inductive proofs about ordinary colourings that might be troublesome without using lists. Most notable, is the proof by \citet{Thomassen-JCTB94} that every planar graph is $5$-choosable. This proof, unlike most proofs of the 5-colourability of planar graphs, does not use the fact that every planar graph has a vertex of degree at most $5$. Given that there are graphs with no $K_t$-minor and minimum degree $\Omega(t\sqrt{\log t})$, this suggests that list colourings might provide an approach for attacking the Hadwiger Conjecture. List colourings also provide a way to handle small separators---first colour one side of the separator, and then colour the second side with the vertices of the separator precoloured. This idea is central in the proofs of \twothmref{RelaxedHadwiger}{Main}. 

We need the following definitions. Let $G$ be a graph. For $A,B\subseteq V(G)$, the pair $\{A,B\}$ is a \emph{separation} of $G$ if $G=G[A]\cup G[B]$ and $A-B\neq\emptyset$ and $B-A\neq\emptyset$. In particular, there is no edge between $A-B$ and $B-A$. The set $A\cap B$ is called a \emph{separator}, and each of $A-B$ and $B-A$ are called \emph{fragments}. If $|A\cap B|\leq t$ then $\{A,B\}$ is called a \emph{$t$-separation} and $A\cap B$ is called a \emph{$t$-separator}. By Menger's Theorem, $G$ is $t$-connected if and only if $G$ has no $(t-1)$-separation and $|V(G)|\geq t+1$. For  $Z\subseteq V(G)$, a separation $\{A,B\}$ of $G$ is \emph{$Z$-good} if $\{A-Z,B-Z\}$ is also a separation of $G-Z$; otherwise it is \emph{$Z$-bad}. Observe that $\{A,B\}$ is \emph{$Z$-bad} if and only if $A-B\subseteq Z$ or $B-A\subseteq Z$.


\thmref{Main} follows from the next lemma (with $Z=\emptyset$ and $L(v)=\{1,\dots,\CEIL{\frac{7t-3}{2}}\}$ for each $v\in V(G)$).

\begin{lemma}
\lemlabel{ListColouring}
Let $G$ be a graph containing no $K_t$-minor. 
Let $Z\subseteq V(G)$ with $|Z|\leq2t-1$. 
Let $L$ be a list assignment of $G$ such that:
\begin{itemize}
\item $|L(v)|=1$ for each vertex $v\in Z$ (said to be ``precoloured''), 
\item $|L(w)|\geq\frac{7t-3}{2}$ for each vertex $w\in V(G)-Z$.
\end{itemize}
Then there is a function $f$ such that:
\begin{enumerate}
\item[(C1)] $f(v)\in L(v)$ for each vertex $v\in V(G)$,
\item[(C2)] for each colour $i$, if $V_i:=\{v\in V(G):f(v)=i\}$ then each component of $G[V_i]$ has at most $N_t+2t-1$ vertices (where $N_t$ comes from \thmref{NT}), and
\item[(C3)] $f(v)\neq f(w)$ for all $v\in Z$ and $w\in N_G(v)-Z$.
\end{enumerate}
\end{lemma}

\begin{proof}
We proceed by induction on $|V(G)|$. 

\textbf{Case I:} First suppose that $|V(G)|\leq N_t+2t-1$. 
For each vertex $v\in Z$, let $f(v)$ be the element of $L(v)$.
For each vertex $w\in V(G)-Z$, choose $f(w)\in L(w)$ such that $f(w)\neq f(v)$ for every vertex $v\in Z$. This is possible since $|L(w)|\geq\frac{7t-3}{2}>2t-1\geq|Z|$. 
Thus (C1) and (C3) are satisfied. (C2) is satisfied since $|V(G)|\leq N_t+2t-1$. 
Now assume that $V(G)\geq N_t+2t-1$. 


\textbf{Case II:} Suppose that some vertex $x\in V(G)-Z$ has degree less than $\frac{7t-3}{2}$ in $G$. 
Let $f$ be the function obtained by induction applied to $G-x$ with $Z$ precoloured.
Choose $f(x)\in L(x)$ such that $f(x)\neq f(y)$ for each $y\in N_G(x)$. 
This is possible since $|L(x)|\geq \frac{7t-3}{2}>\deg(x)$. 
Thus $x$ is in its own monochromatic component. 
Hence (C1), (C2) and (C3) are maintained. 
Now assume that every vertex in $V(G)-Z$ has degree at least $\frac{7t-3}{2}$.

\textbf{Case III:} Suppose that $G$ has a $Z$-good $t$-separation $\{A,B\}$.
Let $P:=Z-B$ and $Q:=Z\cap A\cap B$ and $R:=Z-A$ and $X:=(A\cap B)-Z$.
Thus $P,Q,R,Z$ are pairwise disjoint.
Since $Z=P\cup Q\cup R$, we have $|P|+|Q|+|R|\leq 2t-1$.
Since $A\cap B=Q\cup X$, we have $|Q|+|X|\leq t$ and $|Q|+2|X|\leq 2t$.
Thus  $|P|+|R|+2|Q|+2|X|=(|P|+|Q|+|R|)+(|Q|+2|X|)\leq 4t-1$.
Without loss of generality, $|P|\leq|R|$. 
Thus  $2|P|+2|Q|+2|X|\leq 4t-1$, implying $|P|+|Q|+|X|\leq 2t-1$.
That is, $|A\cap(B\cup Z)|\leq 2t-1$.

Now $B\cup Z\neq V(G)$, as otherwise $A-B\subseteq Z$ and $\{A,B\}$ would be $Z$-bad. Thus the induction hypothesis is applicable to $G[B\cup Z]$ with $Z$ precoloured. (This is why we need to consider $Z$-good and $Z$-bad separations.)\ Hence there is a function $f$ such that:
\begin{enumerate}
\item[(C$1'$)] $f(v)\in L(v)$ for each vertex $v\in B\cup Z$,
\item[(C$2'$)] for each colour $i$, if $V'_i:=\{v\in B\cup Z:f(v)=i\}$ then each component of $G[V'_i]$ has at most $N_t+2t-1$ vertices, and
\item[(C$3'$)] $f(v)\neq f(w)$ for all $v\in Z$ and $w\in (B\cap N_G(v))-Z$.
\end{enumerate} 

Let $L'(w):=\{f(w)\}$ for each vertex $w\in A\cap(B\cup Z)$. 
Let $L'(v):=L(v)$ for each vertex $v\in A-(B\cup Z)$. 
Now apply induction to $G[A]$ with list assignment $L'$, and $A\cap(B\cup Z)$ precoloured. 
This is possible since $|A\cap(B\cup Z)|\leq 2t-1$. Hence there is a function $f$ such that:
\begin{itemize}
\item[(C$1''$)] $f(v)\in L'(v)$ for each vertex $v\in A$,
\item[(C$2''$)] for each colour $i$, if $V''_i:=\{v\in A:f(v)=i\}$ then each component of $G[V''_i]$ has at most $N_t+2t-1$ vertices, and
\item[(C$3''$)] $f(v)\neq f(w)$ for all neighbours $v\in A-(B-Z)$ and $w\in A\cap(B\cup Z)$.
\end{itemize}

Since $L'(v)\subseteq L(v)$, conditions (C$1'$) and (C$1''$) imply (C1). 
Since there is no edge between $A-B$ and $B-A$ in $G$, 
(C$3'$) and (C$3''$) imply that 
every component of $G[V_i]$ is a component of $G[V_i']$ or $G[V_i'']$ or $G[Z]$. 
Since $N_t+2t-1\geq|Z|$, conditions (C$2'$) and (C$2''$) imply (C2).
Hence (C1), (C2) and (C3) are satisfied.
Now assume that every $t$-separation of $G$ is $Z$-bad. 

\textbf{Case IV:} 
Every vertex in $V(G)-Z$ has degree at least $\frac{7t-3}{2}\geq\frac{3}{2}k+|Z|-2$, where $k:=t+1$. 
Thus \thmref{Hope} below implies that $G$ has a $(t+1)$-connected minor $H$ with at least $|V(G)|-|Z|\geq N_t$ vertices. 
By \thmref{NT}, $H$ and thus $G$, has a $K_t$-minor. This contradiction completes the proof.
\end{proof}

\section{Contractibility}
\seclabel{Contractibility}

The main result in this section is \thmref{Hope}, which was used in the proof of \lemref{ListColouring}. The proof reduces to questions about contractibility that are of independent interest. \citet{Mader-DM88} proved the following sufficient condition for a given vertex to be incident to an edge whose contraction maintains connectivity\footnote{\thmref{Mader} is a special case of Theorem 1 in \citep{Mader-DM88} with $\mathfrak{S}=\{\{v,w\}:w\in N_G(v)\}$. Reference \citep{Mader-DM88} cites reference \citep{Mader-AM71} for the proof of Theorem 1 in \citep{Mader-DM88}. The proof of our \thmref{Contractibility} was obtained by following a treatment of Mader's work by \citet{Kriesell-CPC05}.}. See references \citep{Kriesell-GC02, Mader-DM05} for surveys of results in this direction. 

\begin{theorem}[\citet{Mader-DM88}]
\thmlabel{Mader}
Let $v$ be a vertex in a $k$-connected graph $G$, such that every neighbour of $v$ has degree at least $\frac{3}{2}k-1$. Then $G/vw$ is $k$-connected for some edge $vw$ incident to $v$. 
\end{theorem}


The following strengthening of \thmref{Mader} describes a scenario when there is an edge whose contraction \emph{increases} connectivity.

\begin{theorem}
\thmlabel{Contractibility}
Let $v$ be a vertex in graph $G$, such that $N_G(v)$ is the only minimal $(k-1)$-separator, 
and every neighbour of $v$ has degree at least $\frac{3}{2}k-1$. Then $G/vw$ is $k$-connected for some edge $vw$ incident to $v$.
\end{theorem}

The condition in \thmref{Contractibility} is equivalent to saying that every $(k-1)$-separation of $G$ is $\{v\}$-bad. 
Thus \thmref{Contractibility} is a special case of the following theorem  (with $Z=\{v\}$).

\begin{theorem}
\thmlabel{Hope}
Suppose that $G$ is a graph and for some $Z\subset V(G)$,
\begin{itemize}
\item every $(k-1)$-separation of $G$ is $Z$-bad, and
\item every vertex in $\cup\{N_G(v)-Z:v\in Z\}$ has degree at least $\frac{3}{2}k+|Z|-2$ in $G$.
\end{itemize}
Then $G$ has a set of at most $|Z|$ edges, each with one endpoint in $Z$, whose contraction gives a $k$-connected graph.
\end{theorem}

\begin{proof}
We proceed by induction on $|Z|$.
If $Z=\emptyset$, or $N_G(v)\subseteq Z$ for each $v\in Z$, then $G-Z$ is $k$-connected.
Now assume that $N_G(v)\not\subseteq Z$ for some $v\in Z$.
By assumption, every vertex in $N_G(v)-Z$ has degree at least $\frac{3}{2}k+|Z|-2$ in $G$.
By \lemref{Induction} below there is an edge $vw$ with $w\in N_G(v)-Z$, 
such that  every $(k-1)$-separation of $G/vw$ is $(Z-\{v\})$-bad.
For every vertex $x\in V(G/vw)$,
If contracting $vw$ decreases the degree of some vertex $x$, then $x$ is a common neighbour of $v$ and $w$, and $\deg_{G/vw}(x)=\deg_G(x)-1$. Thus $\deg_{G/vw}(x)\geq\frac{3}{2}k+|Z-\{v\}|-2$.
By induction, $G/vw$ has a set $S$ of at most $|Z-\{v\}|$ edges whose contraction gives a $k$-connected graph.
Thus $S\cup\{vw\}$ is a set of at most $|Z|$ edges in $G$ whose contraction gives a $k$-connected graph.
\end{proof}

\begin{lemma}
\lemlabel{Induction}
Suppose that $G$ is a graph and for some $Z\subset V(G)$ and for some vertex $v\in Z$ with $N_G(v)-Z\neq\emptyset$,
\begin{itemize}
\item every $(k-1)$-separation of $G$ is $Z$-bad, and
\item every vertex in $N_G(v)-Z$ has degree at least $\frac{3}{2}k+|Z|-2$ in $G$.
\end{itemize}
Then there is an edge $vw$ with $w\in N_G(v)-Z$, such that 
\begin{itemize}
\item every $(k-1)$-separation of $G/vw$ is $(Z-\{v\})$-bad.
\end{itemize}
\end{lemma}

\begin{proof}

Suppose on the contrary that for each $w\in N_G(v)-Z$, the contracted graph $G/vw$ has a $(Z-\{v\})$-good $(k-1)$-separator. This separator must contain the vertex obtained by contracting $vw$. Thus $G$ has a $Z$-good $k$-separator containing $v$ and $w$.
Let \seps\ be the set of $Z$-good $k$-separations $\{A,B\}$ of $G$ such that $v\in A\cap B$ and $A\cap B\cap (N_G(v)-Z)\neq\emptyset$. We say $\{A,B\}\in\seps$ \emph{belongs} to $x$ for each $x\in A\cap B\cap (N_G(v)-Z)$.
As proved above, for each $w\in N_G(v)-Z$, some separation in \seps\ belongs to $w$.


For each separation $\{A,B\}\in\seps$,
\begin{equation}
\eqnlabel{FragmentContainsNeighbour}
(A-B)\cap(N_G(v)-Z)\neq\emptyset\quad\text{and}\quad
(B-A)\cap (N_G(v)-Z)\neq\emptyset\enspace;
\end{equation}
otherwise $\{A-\{v\},B\}$ or $\{A,B-\{v\}\}$ would be a $Z$-good $(k-1)$-separation of $G$.


Say $\{A,B\}\in\seps$ belongs to $x\in N_G(v)-Z$, 
and $\{C,D\}\in\seps$ belongs to $y\in (N_G(v)-Z)-\{x\}$.
Let $S:=A\cap B$ and $T:=C\cap D$ be the corresponding separators in $G$.
Let $A':=A-B$ and $B':=B-A$ and $C':=C-D$ and $D':=D-C$ be the corresponding fragments in $G$.
Let $U:=(S\cap{C'})\cup(S\cap T)\cup (T\cap{A'})$. 
Thus $U$ separates ${A'}\cap{C'}$ and $B'\cup D'$, as illustrated in \figref{Separators}.

\begin{figure}[h]
\begin{center}
\includegraphics{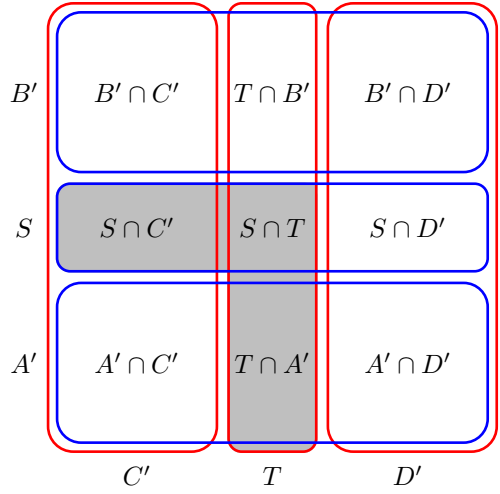}
\caption{\figlabel{Separators} Separator $S$ and its fragments $A'$ and $B'$.
Separator $T$ and its fragments $C'$ and $D'$.
The induced separator $U$ is shaded.}
\end{center}
\end{figure}


Suppose that ${A'}\cap{C'}\not\subseteq Z$.
Since $\{A,B\}$ is $Z$-good, $B'\not\subseteq Z$.
Since $B'\cup D'\not\subseteq Z$, 
$$\mathcal{U}:=\big\{({A'}\cap{C'})\cup U,{B'}\cup{D'}\cup U\big\}$$ 
is a $Z$-good separation of $G$, whose separator is $U$.
Thus $|U|\geq k$.
That is, $|S\cap{C'}|+|S\cap T|+|T\cap{A'}|\geq k$.
Now $|S\cap{C'}|+|S\cap T|=|S|-|S\cap{D'}|\leq k-|S\cap{D'}|$.
Hence $k-|S\cap{D'}|+|T\cap{A'}|\geq k$, implying $|T\cap{A'}|\geq|S\cap{D'}|$. 
Similarly, $|S\cap{C'}|\geq|T\cap{B'}|$.
By symmetry:
\begin{align}
A'\cap {C'}\not\subseteq Z\quad\Longrightarrow\quad&
|T\cap {A'}|\geq|S\cap{D'}|
\quad\text{and}\quad
|S\cap {C'}|\geq|T\cap{B'}|\eqnlabel{BlahOne}\\
A'\cap D'\not\subseteq Z\quad\Longrightarrow\quad&
|T\cap {A'}|\geq|S\cap{C'}|
\quad\text{and}\quad
|S\cap{D'}|\geq|T\cap{B'}|\eqnlabel{BlahTwo}\\
B'\cap C'\not\subseteq Z\quad\Longrightarrow\quad&
|T\cap {B'}|\geq|S\cap{D'}|
\quad\text{and}\quad
|S\cap {C'}|\geq|T\cap{{A'}}|\eqnlabel{BlahThree}\\
B'\cap D'\not\subseteq Z\quad\Longrightarrow\quad&
|T\cap{B'}|\geq|S\cap{{C'}}|
\quad\text{and}\quad
|S\cap{D'}|\geq|T\cap{{A'}}|\eqnlabel{BlahFour}
\end{align}

Choose a separation $\{A,B\}\in\seps$ that minimises $\min\{|A-B|,|B-A|\}$. 
Let $x$ be a vertex in $N_G(v)-Z$ such that $\{A,B\}$ belongs to $x$.
Define the separator $S$, and the fragments $A'$ and $B'$ as above.
Without loss of generality, $|A'|\leq|B'|$. 
By \eqnref{FragmentContainsNeighbour}, there is a vertex $y\in (N_G(v)-Z)\cap{A-B}$.
Let $\{C,D\}$ be a separator in \seps\ that belongs to $y$.
Define the separator $T$, and the fragments $C'$ and $D'$ as above.

Suppose that $A'\cap C'\not\subseteq Z$ and $B'\cap D'\not\subseteq Z$.
By \twoeqnref{BlahOne}{BlahFour}, $|T\cap A'|=|S\cap D'|$. 
Define $U$ and $\mathcal{U}$ as above.
Thus $\mathcal{U}$ is a $Z$-good separation of $G$, whose separator is $U$.
Now $|U|=|S\cap{C'}|+|S\cap T|+|S\cap{D'}|=|S|\leq k$.
Thus $\mathcal{U}$ is a $Z$-good $k$-separation.
Observe that $v\in S\cap T\subseteq U$ and $y\in{A'}\cap T\subseteq U$.
Thus $\mathcal{U}\in\seps$ and $\mathcal{U}$ belongs to $y$.
One fragment of $\mathcal{U}$ is ${A'}\cap{C'}\subseteq A_s-\{y\}$ since $y\in T$.
Thus $|{A'}\cap{C'}|<|A_s|$, which contradicts the choice of $\{A,B\}$.

Thus $A'\cap C'\subseteq Z$ or $B'\cap D'\subseteq Z$.
By symmetry,  $A'\cap D'\subseteq Z$ or $B'\cap C'\subseteq Z$.
It follows that 
${A'}-Z\subseteq T$ or ${B'}-Z\subseteq T$ or 
${C'}-Z\subseteq S$ or ${D'}-Z\subseteq S$.
The choice of $S$ will not be used in the remainder. 
So without loss of generality, $A'-Z\subseteq T$.

We claim that $A'-Z$ or $B'-Z$ or $C'-Z$ or $D'-Z$ has at most $\half\max\{|S-T|,|T-S|\}$ vertices.
If $B'-Z\subseteq T$, then $(A'-Z)\cup(B'-Z)\subseteq T-S$, implying 
$A_S-Z$ or $B'-Z$ has at most $\half|T-S|$ vertices, as claimed.
Now assume that $B'-Z\not\subseteq T$.
Without loss of generality, $B'\cap C'\not\subseteq Z$.
By \eqnref{BlahThree}, $|S\cap C'|\geq|T\cap{A'}|=|A'-Z|$.
If $|A'-Z|\leq\half|S-T|$ then the claim is proved.
Otherwise, $|S\cap C'|\geq|A'-Z|>\half|S-T|$.
Thus $|S\cap{D'}|<\half|S-T|$ (since $S-T$ is the disjoint union of $S\cap{C'}$ and $S\cap{D'}$).
If $D'-Z\subseteq S$ then $|D'-Z|\leq |D'\cap S|<\half|S-T|$.
So assume that $D'-Z\not\subseteq S$. Thus ${D'}\cap{B'}\not\subseteq Z$.
By \eqnref{BlahFour}, 
$|S\cap{D'}|\geq|T\cap{A'}|=|A'-Z|>\half|S-T|$, which is a contradiction.

Hence $|Q-Z|\leq\half\max\{|S-T|,|T-S|\}$ for some fragment $Q\in\{A',B',C', D'\}$.
Now $\max\{|S-T|,|T-S|\}=\max\{|S|,|T|\}-|S\cap T|\leq k-1$ since $v\in S\cap T$.
Thus $|Q|\leq\half(k-1)+|Z|$.
By \eqnref{FragmentContainsNeighbour}, there is a vertex $w\in (N_G(v)-Z)\cap Q$.
Then $N_G(w)\subseteq Q\cup S$ or $N_G(w)\subseteq Q\cup T$.
Since $v\in S\cap T\cap Z$ and $|S-\{v\}|\leq k-1$ and $|T-\{v\}|\leq k-1$ and  $w\in Q$, we have
$\deg(w)\leq \half(k-1)+|Z|+(k-1)-1=\frac{3k-5}{2}+|Z|$.
This contradicts the assumption that each vertex in $N_G(v)-Z$ has degree at least $\frac{3}{2}k+|Z|-2$.
\end{proof}


We now show that the degree bound in \thmref{Contractibility} is best possible. The proof is an adaptation of a construction by \citet{Watkins-JCT70} that shows that the degree bound in \thmref{Mader} is best possible. For odd $k\geq5$ and $n\in[4,k-1]$, let $p:=\half(k-1)$. Start with the lexicographic product $C_n\cdot K_p$, which consists of $n$ disjoint copies $H_1,\dots,H_n$ of $K_p$, where every vertex in $H_i$ is adjacent to every vertex in $H_{i+1}$, and $H_j$ means $H_{j\bmod{n}}$. Let $G$ be the graph obtained by adding a new vertex $v$ adjacent to one vertex $w_i$ in each $H_i$, as illustrated in \figref{Construction}. It is straightforward to verify  that there are $k$ internally disjoint paths in $G$ between each pair of distinct vertices in $V(G)-\{v\}$. Thus $N_G(v)$ is the only minimal $(k-1)$-separator in $G$ (since $\deg(v)=n\leq k-1$). For each  neighbour $w_i$ of $v$, observe that  $\deg(w_i)=(p-1)+2p+1=\frac{3}{2}(k-1)$, but in $G/vw$ the set $V(H_i)\cup V(H_{i+2})$ is a $2p$-separator, implying $G/vw_i$ is not $k$-connected. Thus the degree bound of $\frac{3}{2}k-1$ in \thmref{Contractibility} is best possible.

\begin{figure}
\begin{center}
\includegraphics{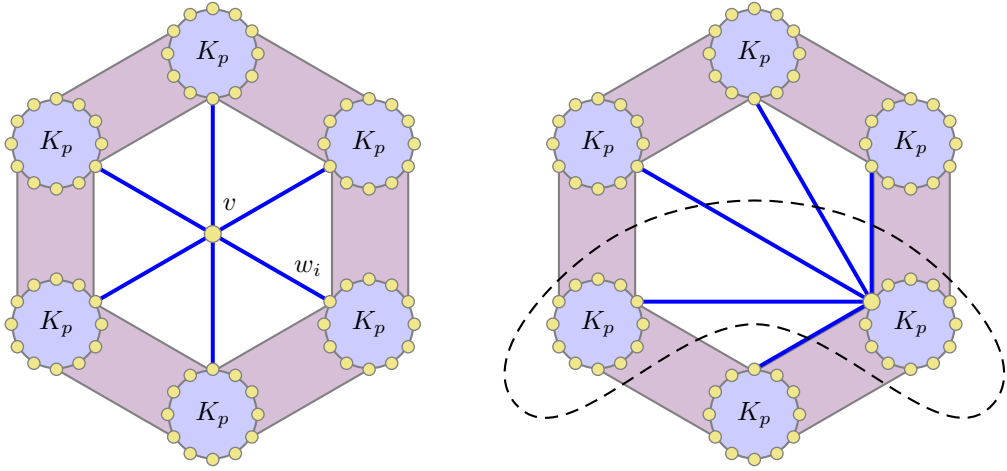}
\end{center}
\caption{\figlabel{Construction} Contracting $vw_i$ produces a $(k-1)$-separation.}
\end{figure}

\section{Final Remarks}

Seymour and Thomas conjectured the following strengthening of \thmref{NT}.

\begin{conjecture}[Seymour and Thomas] 
\conjlabel{Apex}
For each $t\in\mathbb{Z}^+$ there exists $N_t\in\mathbb{Z}^+$ such that every $t$-connected graph $G$ with at least $N_t$ vertices and no $K_t$ minor contains a set $S$ of $t-5$ vertices such that $G-S$ is planar.
\end{conjecture}

\citet{KNRWb,KNRWa} proved this conjecture for $t\leq 6$. Recently, \citet{NorineThomas} proved it for $t\leq8$. If true, 
\conjref{Apex} could be used instead of \thmref{NT} to make small improvements to \thmref{Main}.

\smallskip
Given that list colourings are a useful tool in attacking the Hadwiger Conjecture, it is interesting to ask what is the least function $f$ such that every graph with no $K_t$-minor is $f(t)$-choosable. Since every graph with no $K_t$-minor has a vertex of degree at most $ct\sqrt{\log t}$, it follows  that $f(t)\leq ct\sqrt{\log t}$, and this is the best known bound. In particular, the following conjecture is unsolved.

\medskip\noindent\textbf{Weak List Hadwiger Conjecture.}
For some constant $c$, 
every graph with no $K_t$-minor is $ct$-choosable.
\smallskip

\citet{KawaMohar-JCTB07} discuss this conjecture and suggest that it might be true with $c=\frac{3}{2}$. We dare to suggest the following.

\medskip\noindent\textbf{List Hadwiger Conjecture.}
Every graph with no $K_t$-minor is $t$-choosable.
\smallskip

This conjecture is true for $t\leq 5$ \citep{Skrekovski-DM98,HMS-DM08}. The $t=6$ case is open. 

\section*{Acknowledgements}

Thanks to Ga\v{s}per Fijav\v{z}, Vida Dujmovi{\'c}, Matthias Kriesell, and Attila P{\'o}r for helpful discussions.


\begin{thebibliography}{29}
\providecommand{\natexlab}[1]{#1}
\providecommand{\url}[1]{\texttt{#1}}
\providecommand{\urlprefix}{}
\expandafter\ifx\csname urlstyle\endcsname\relax
  \providecommand{\doi}[1]{doi:\discretionary{}{}{}#1}\else
  \providecommand{\doi}{doi:\discretionary{}{}{}\begingroup
  \urlstyle{rm}\Url}\fi

\bibitem[{B{\"o}hme et~al.(ress)B{\"o}hme, ichi Kawarabayashi, Maharry, and
  Mohar}]{BKMM}
\textsc{Thomas B{\"o}hme, Ken ichi Kawarabayashi, John Maharry, and Bojan
  Mohar}.
\newblock Linear connectivity forces large complete bipartite minors.
\newblock \emph{J. Combin. Theory Ser. B}, in press.

\bibitem[{Dirac(1952)}]{Dirac52}
\textsc{Gabriel~A. Dirac}.
\newblock A property of {$4$}-chromatic graphs and some remarks on critical
  graphs.
\newblock \emph{J. London Math. Soc.}, 27:85--92, 1952.

\bibitem[{Erd{\H{o}}s et~al.(1980)Erd{\H{o}}s, Rubin, and Taylor}]{ERT80}
\textsc{Paul Erd{\H{o}}s, Arthur~L. Rubin, and Herbert Taylor}.
\newblock Choosability in graphs.
\newblock In \emph{Proc.\ of the West Coast Conference on Combinatorics, Graph
  Theory and Computing}, vol. XXVI of \emph{Congress. Numer.}, pp. 125--157.
  Utilitas Math., 1980.

\bibitem[{Gonthier(2008)}]{Gonthier-NAMS08}
\textsc{Georges Gonthier}.
\newblock Formal proof---the four-color theorem.
\newblock \emph{Notices Amer. Math. Soc.}, 55(11):1382--1393, 2008.

\bibitem[{Hadwiger(1943)}]{Hadwiger43}
\textsc{Hugo Hadwiger}.
\newblock \"{U}ber eine {K}lassifikation der {S}treckenkomplexe.
\newblock \emph{Vierteljschr. Naturforsch. Ges. Z\"urich}, 88:133--142, 1943.

\bibitem[{He et~al.(2008)He, Miao, and Shen}]{HMS-DM08}
\textsc{Wenjie He, Wenjing Miao, and Yufa Shen}.
\newblock Another proof of the 5-choosability of {$K\sb 5$}-minor-free graphs.
\newblock \emph{Discrete Math.}, 308(17):4024--4026, 2008.

\bibitem[{Kawarabayashi(2008)}]{Kawarabayashi-CPC08}
\textsc{Ken-ichi Kawarabayashi}.
\newblock A weakening of the odd {H}adwiger's conjecture.
\newblock \emph{Combin. Probab. Comput.}, 17:815--821, 2008.

\bibitem[{Kawarabayashi and Mohar(2007)}]{KawaMohar-JCTB07}
\textsc{Ken-ichi Kawarabayashi and Bojan Mohar}.
\newblock A relaxed {H}adwiger's conjecture for list colorings.
\newblock \emph{J. Combin. Theory Ser. B}, 97(4):647--651, 2007.

\bibitem[{Kawarabayashi et~al.(2005)Kawarabayashi, Norine, Thomas, and
  Wollan}]{KNRWa}
\textsc{Ken-ichi Kawarabayashi, Serguei Norine, Robin Thomas, and Paul Wollan}.
\newblock {$K_6$} minors in large $6$-connected graphs of bounded treewidth,
  2005.
\newblock Submitted.

\bibitem[{Kawarabayashi et~al.(2008)Kawarabayashi, Norine, Thomas, and
  Wollan}]{KNRWb}
\textsc{Ken-ichi Kawarabayashi, Serguei Norine, Robin Thomas, and Paul Wollan}.
\newblock {$K_6$} minors in large $6$-connected graphs, 2008.
\newblock In preparation.

\bibitem[{Kostochka(1982)}]{Kostochka82}
\textsc{Alexandr~V. Kostochka}.
\newblock The minimum {H}adwiger number for graphs with a given mean degree of
  vertices.
\newblock \emph{Metody Diskret. Analiz.}, 38:37--58, 1982.

\bibitem[{Kostochka(1984)}]{Kostochka84}
\textsc{Alexandr~V. Kostochka}.
\newblock Lower bound of the {H}adwiger number of graphs by their average
  degree.
\newblock \emph{Combinatorica}, 4(4):307--316, 1984.

\bibitem[{Kriesell(2002)}]{Kriesell-GC02}
\textsc{Matthias Kriesell}.
\newblock A survey on contractible edges in graphs of a prescribed vertex
  connectivity.
\newblock \emph{Graphs Combin.}, 18(1):1--30, 2002.

\bibitem[{Kriesell(2005)}]{Kriesell-CPC05}
\textsc{Matthias Kriesell}.
\newblock Triangle density and contractability.
\newblock \emph{Combin. Probab. Comput.}, 14(1-2):133--146, 2005.

\bibitem[{Mader(1971)}]{Mader-AM71}
\textsc{Wolfgang Mader}.
\newblock Eine {E}igenschaft der {A}tome endlicher {G}raphen.
\newblock \emph{Arch. Math. (Basel)}, 22:333--336, 1971.

\bibitem[{Mader(1988)}]{Mader-DM88}
\textsc{Wolfgang Mader}.
\newblock Generalizations of critical connectivity of graphs.
\newblock \emph{Discrete Math.}, 72(1-3):267--283, 1988.

\bibitem[{Mader(2005)}]{Mader-DM05}
\textsc{Wolfgang Mader}.
\newblock High connectivity keeping sets in graphs and digraphs.
\newblock \emph{Discrete Math.}, 302(1-3):173--187, 2005.

\bibitem[{Norine and Thomas(2008)}]{NorineThomas}
\textsc{Serguei Norine and Robin Thomas}.
\newblock {$K_t$}-minors, 2008.
\newblock Presented at the Banff Graph Minors Workshop.

\bibitem[{Robertson et~al.(1997)Robertson, Sanders, Seymour, and
  Thomas}]{RSST97}
\textsc{Neil Robertson, Daniel~P. Sanders, Paul~D. Seymour, and Robin Thomas}.
\newblock The four-colour theorem.
\newblock \emph{J. Combin. Theory Ser. B}, 70(1):2--44, 1997.

\bibitem[{Robertson et~al.(1993)Robertson, Seymour, and Thomas}]{RST-Comb93}
\textsc{Neil Robertson, Paul~D. Seymour, and Robin Thomas}.
\newblock Hadwiger's conjecture for ${K}\sb 6$-free graphs.
\newblock \emph{Combinatorica}, 13(3):279--361, 1993.

\bibitem[{{\v{S}}krekovski(1998)}]{Skrekovski-DM98}
\textsc{Riste {\v{S}}krekovski}.
\newblock Choosability of {$K\sb 5$}-minor-free graphs.
\newblock \emph{Discrete Math.}, 190(1--3):223--226, 1998.

\bibitem[{Thomason(1984)}]{Thomason84}
\textsc{Andrew Thomason}.
\newblock An extremal function for contractions of graphs.
\newblock \emph{Math. Proc. Cambridge Philos. Soc.}, 95(2):261--265, 1984.

\bibitem[{Thomason(2001)}]{Thomason01}
\textsc{Andrew Thomason}.
\newblock The extremal function for complete minors.
\newblock \emph{J. Combin. Theory Ser. B}, 81(2):318--338, 2001.

\bibitem[{Thomassen(1994)}]{Thomassen-JCTB94}
\textsc{Carsten Thomassen}.
\newblock Every planar graph is {$5$}-choosable.
\newblock \emph{J. Combin. Theory Ser. B}, 62(1):180--181, 1994.

\bibitem[{Toft(1996)}]{Toft-HadwigerSurvey96}
\textsc{Bjarne Toft}.
\newblock A survey of {H}adwiger's conjecture.
\newblock \emph{Congr. Numer.}, 115:249--283, 1996.

\bibitem[{Vizing(1976)}]{Vizing76}
\textsc{V.~G. Vizing}.
\newblock Coloring the vertices of a graph in prescribed colors.
\newblock \emph{Metody Diskret. Analiz}, 29:3--10, 1976.

\bibitem[{Wagner(1937)}]{Wagner37}
\textsc{Klaus Wagner}.
\newblock {\"U}ber eine {E}igenschaft der ebene {K}omplexe.
\newblock \emph{Math. Ann.}, 114:570--590, 1937.

\bibitem[{Watkins(1970)}]{Watkins-JCT70}
\textsc{Mark~E. Watkins}.
\newblock Connectivity of transitive graphs.
\newblock \emph{J. Combinatorial Theory}, 8:23--29, 1970.

\bibitem[{Woodall(2001)}]{Woodall-ListColouringSurvey}
\textsc{Douglas~R. Woodall}.
\newblock List colourings of graphs.
\newblock In \emph{Surveys in combinatorics}, vol. 288 of \emph{London Math.
  Soc. Lecture Note Ser.}, pp. 269--301. Cambridge Univ. Press, 2001.

\end{thebibliography}

\def\soft#1{\leavevmode\setbox0=\hbox{h}\dimen7=\ht0\advance \dimen7
  by-1ex\relax\if t#1\relax\rlap{\raise.6\dimen7
  \hbox{\kern.3ex\char'47}}#1\relax\else\if T#1\relax
  \rlap{\raise.5\dimen7\hbox{\kern1.3ex\char'47}}#1\relax \else\if
  d#1\relax\rlap{\raise.5\dimen7\hbox{\kern.9ex \char'47}}#1\relax\else\if
  D#1\relax\rlap{\raise.5\dimen7 \hbox{\kern1.4ex\char'47}}#1\relax\else\if
  l#1\relax \rlap{\raise.5\dimen7\hbox{\kern.4ex\char'47}}#1\relax \else\if
  L#1\relax\rlap{\raise.5\dimen7\hbox{\kern.7ex
  \char'47}}#1\relax\else\message{accent \string\soft \space #1 not
  defined!}#1\relax\fi\fi\fi\fi\fi\fi} \def\Dbar{\leavevmode\lower.6ex\hbox to
  0pt{\hskip-.23ex \accent"16\hss}D}

\end{document}